\documentclass[11pt]{amsart}
\usepackage{amsmath}
\usepackage{amssymb}
\usepackage{amsfonts}
\usepackage{mathrsfs}
\usepackage{graphicx}
\usepackage{texdraw}
\usepackage{graphpap}
\usepackage{wasysym}
\usepackage{enumitem}
\usepackage{color}
\usepackage[OT2,T1]{fontenc}

\theoremstyle{plain}
\newtheorem{thm}{Theorem}[section]
\newtheorem{mth}{Theorem}

\newtheorem{cor}[thm]{Corollary}
\newtheorem{lem}[thm]{Lemma}

\theoremstyle{remark}

\newtheorem*{rem}{Remark}

\newtheorem*{ex}{Example}

\numberwithin{equation}{section}

\newcommand{\con}{\alpha}
\newcommand{\num}{\gamma}
\newcommand{\para}{\nu}
\newcommand{\loc}{\operatorname{loc}}

\newcommand{\bfR}{\mathbf{R}}

\newcommand{\calH}{{\mathscr H}}

\newcommand{\calP}{{\mathscr P}}

\newcommand{\bfL}{{\mathbf L}}
\newcommand{\bfk}{{\mathbf k}}

\newcommand{\Int}{\operatorname{Int}}

\newcommand{\C}{{\mathbb C}}

\newcommand{\const}{\mathrm{const.}}
\newcommand{\eps}{{\varepsilon}}

\newcommand{\re}{\operatorname{Re}}
\newcommand{\im}{\operatorname{Im}}

\newcommand{\Poly}{\operatorname{Pol}}

\renewcommand{\d}{{\partial}}
\newcommand{\dbar}{\bar{\partial}}

\newcommand{\supp}{\operatorname{supp}}

\newcommand{\Lap}{\Delta}
\newcommand{\lnorm}{\left\|}
\newcommand{\rnorm}{\right\|}

\def\norm#1{\lnorm {#1} \rnorm}

\def\labs{\left |}
\def\rabs{\right |}
\def\babs#1{\labs {#1} \rabs}

\begin{document}

\title{Microscopic densities and Fock-Sobolev spaces}

\subjclass[2010]{30H20; 60B20}

\begin{abstract} We study two-dimensional eigenvalue ensembles close to certain types of singular points
in the interior of the droplet. We prove existence of a microscopic density which quickly approaches the equilibrium density, as the distance from the singularity increases beyond the microscopic scale.
This kind of asymptotic is used to analyze normal matrix models in \cite{AKS}.
 In addition, we obtain here asymptotics for the Bergman function of certain Fock-Sobolev spaces of entire functions. \end{abstract}

\author{Yacin Ameur}

\address{Yacin Ameur\\
Department of Mathematics\\
Faculty of Science\\
Lund University\\
P.O. BOX 118\\
221 00 Lund\\
Sweden}

\email{Yacin.Ameur@maths.lth.se}

\author{Seong-Mi Seo}

\address{Seong-Mi Seo\\ School of Mathematics\\ Korea Institute for Advanced Study\\
85 Hoegiro\\ Dongdaemun-gu\\ Seoul 02455\\ Republic of Korea}

\email{seongmi@kias.re.kr}

\keywords{Microscopic density; Fock-Sobolev space; Bergman function}

\thanks{Seo was supported by Samsung Science and Technology Foundation, SSTF-BA1401-01.}

\maketitle

\section{Introduction and main results}

\subsection{Microscopic potentials} \label{micropot}
Consider a real-valued polynomial
$Q_0(z)=Q_0(z,\bar{z})$,
positively homogeneous of some even degree $2k$, where $k\ge 1$. We assume that $Q_0$ be positive definite, i.e., $Q_0(z)>0$ when $z\ne 0$. With a minor restriction, we
will also assume that
\begin{equation}\label{norma}\d^{2k}Q_0(0)=0.\end{equation}

Finally, we fix a real parameter $c>-1$ and put
\begin{equation}\label{vdef}V_0(z)=Q_0(z)-2c\log|z|.\end{equation}
We call $V_0$ a \textit{microscopic potential.}

Here and throughout, we write
$\d=(\d_x-i\d_y)/2$ and $\dbar=(\d_x+i\d_y)/2$. By $\Lap:=\d\dbar$ we denote $1/4$ times the standard Laplacian on $\C$. We write $dA=dxdy/\pi$ for Lebesgue measure divided by $\pi$.

\subsection{The Bergman function} Consider the measure $d\mu_0=e^{-V_0}\, dA$, and
let $L^2_a(\mu_0)$ be the Bergman space of all entire functions $u$ such that
\begin{equation}\label{fsh}\|u\|_{L^2(\mu_0)}^2:=\int_\C\babs{u}^2\, d\mu_0=\int_\C\babs{u(z)}^2|z|^{2c}e^{-Q_0(z)}\, dA(z)<\infty.
\end{equation}
We will write $L_0(z,w)$ for the Bergman kernel in $L^2_a(\mu_0)$. The main object of interest for the present investigation is the
\textit{Bergman function} of $L^2_a(\mu_0)$,
$$R_0(z)=L_0(z,z)e^{-V_0(z)}.$$

The
space $L^2_a(\mu_0)$ is a kind of Fock-Sobolev space of entire functions, associated with the generalized Fock-weight $Q_0$ and the Sobolev parameter $c$.
Incidentally, related
spaces were introduced recently
in the papers \cite{CCK,CZ}.

Our first result is the following.

\begin{mth} \label{MT1} There exists a constant $\con=\con[Q_0]>0$ such that
\begin{equation}\label{mil}R_0(z)=\Lap Q_0(z)\cdot (1+O(e^{-\con|z|^{2k}})),\qquad \text{as}\quad  |z|\to\infty.\end{equation}
\end{mth}

\begin{ex} Let $\lambda$ be a positive parameter and consider the "Mittag-Leffler potential''
$$V_0=|z|^{2\lambda}-2c\log|z|,$$
we have explicitly $R_0(z)=E(|z|^2)e^{-V_0(z)}$ where $E$ is given by
\begin{equation}\label{opf}E(z)=\lambda\cdot E_{1/\lambda,(1+c)/\lambda}(z).\end{equation}
Here $E_{a,b}$ is the Mittag-Leffler function (cf. \cite{GKMR})
$$E_{a,b}(z)=\sum_{j=0}^\infty \frac {z^j}{\Gamma\left(a j +b\right)}.$$
The formula \eqref{opf} can be verified by direct calculation of the orthonormal polynomials $e_j$ with respect to the measure $\mu_0$,
since $L_0(z,w)=\sum_0^\infty e_j(z)\bar{e}_j(w)$.

The asymptotic estimate for Mittag-Leffler functions in \cite[eq.(4.7.4),\,p.75]{GKMR} shows that, as $z\to\infty$, (if $\lambda>1/2$)
$$R_0(z)=
\lambda^2|z|^{2(\lambda-1)}+O(|z|^{-2-2c}e^{-|z|^{2\lambda}}).$$
Hence if we take $\lambda$ to be an integer $k$, then \eqref{mil} holds with the sharp error term
$$R_0(z)=\Lap Q_0(z)\cdot(1+O(|z|^{2(-c-k)}e^{-|z|^{2k}})).$$ An asymptotic expansion for the $O$-term is found in \cite{AKS}, cf. \cite{GKMR}.
\end{ex}

\begin{rem} Temporarily drop the assumption \eqref{norma} and write $$\kappa=\frac {\d^{2k}Q_0(0)}{(2k)!}.$$
Theorem \ref{MT1} still holds when $\kappa\ne 0$, provided that the other assumptions on $Q_0$ are satisfied.

Indeed if $\tilde{Q}_0=Q_0-2\re(\kappa z^{2k})$ then $\d^{2k}\tilde{Q}_0(0)=0$, so we can apply Theorem \ref{MT1} to conclude that the Bergman function corresponding to $\tilde{V}_0:=\tilde{Q}_0-2c\log|z|$ satisfies
$\tilde{R}_0(z)=\Lap Q_0(z)\cdot (1+O(e^{-\alpha|z|^2}))$ as $z\to\infty$. Then, using the isometric isomorphism
$$L^2_a(\tilde{\mu}_0)\to L^2_a(\mu_0)\quad ,\quad u(z)\mapsto u(z)\cdot e^{\kappa z^{2k}},$$
 one deduces readily that the Bergman kernels pertaining to $V_0$ and $\tilde{V}_0$ are related via
$L_0(z,w)=\tilde{L}_0(z,w)e^{\kappa z^{2k}+\bar{\kappa}\bar{w}^{2k}}$. Passing to Bergman functions, we conclude that
$R_0(z)=\Lap Q_0(z)\cdot (1+O(e^{-\alpha|z|^2}))$ as $z\to\infty$, as desired.
\end{rem}

\subsection{Microscopic densities} \label{mde} Let $Q$ be a suitable real-valued "potential function'', of sufficient growth near $\infty$. (Precisely, $\liminf_{\zeta\to\infty}\frac {Q(\zeta)}{\log|\zeta|}>2$.)

We associate with $Q$ the \textit{equilibrium measure} $\sigma$, which is the unique compactly supported Borel probability measure which minimizes the weighted logarithmic energy $$I_Q[\mu]=\iint_{\C^2}\log\frac 1 {|\zeta-\eta|}\,d\mu(\zeta)d\mu(\eta)+\int_\C Q\, d\mu.$$
It is well-known \cite{ST} that $\sigma$ takes the form
\begin{equation}\label{eqm}d\sigma=\Lap Q\cdot\chi_S\cdot dA,\end{equation} where the support $S=\supp\sigma$ is a compact set which one calls the \textit{droplet} in external field $Q$.

We assume in the following that $0$ is in the interior of $S$. We shall also assume that $Q$ is real-analytic
in a neighbourhood of $0$ and that the Taylor expansion of $\Lap Q$ about $0$ takes the form
$$\Lap Q=P+\text{"higher\,order\, terms''},$$
where $P$ is positive definite and homogeneous of degree $2k-2$.
Following \cite{AS}, we write $\tilde{Q}_0$ for the Taylor polynomial of degree $2k$ of $Q$ about $0$ and
put $Q_1=Q-\tilde{Q}_0$. We then introduce functions $H$ and $Q_0$ by
$$H(\zeta):=Q(0)+2\d Q(0)\zeta+\cdots+\frac 2 {(2k)!}\d^{2k}{Q}(0)\zeta^{2k},\quad Q_0:=\tilde{Q}_0-H.$$
Then we have \begin{equation}\label{candec}Q=Q_0+\re H+Q_1,\end{equation} where $Q_0$ is homogeneous of degree $2k$ and $Q_1=O(|\zeta|^{2k+1})$ as $\zeta\to0$. We refer to \eqref{candec} as the \textit{canonical decomposition} of the potential $Q$ about $0$.

It is useful to allow further generality, by adding to the potential
a term $h(\zeta)/n$ where $h$ might have finitely many logarithmic singularities at distinct points $a_1,\cdots,a_s$ in the punctured plane
$\C^*:=\C\setminus\{0\}$,
\begin{equation}\label{distinct}h(\zeta)=h_0(\zeta)+\sum_{j=1}^{s} 2c_j \log |\zeta-a_j|,\quad c_j >-1.\end{equation}
The function $h_0$ is any fixed $C^{1,1}$-smooth, real-valued function on $\C$ which satisfies (almost everywhere) estimates of the form
\begin{equation}\label{slow}h_0(\zeta)\le C[1+\log(1+|\zeta|^2)],\quad \Lap h_0(\zeta)\le C(1+|\zeta|^2)^{-2}.
\end{equation}

We finally fix $c>-1$ and form the $n$-dependent potential
\begin{equation}\label{gnoll}V_n=Q-2(c/n) \ell-(1/n)h,\qquad (\ell(\zeta):=\log|\zeta|).\end{equation}
Adding an $n$-dependent constant to $Q$ does not lead to any essential changes, so we assume in the
following that
$$Q(0)=H(0)=h(0)=0.$$

It is here convenient to recall a few notions from the theory of two-dimensional eigenvalue ensembles. (The reader who is not familiar with this theory is insured that our main results as well as our arguments can be understood without it.)

Consider a system ("eigenvalue ensemble'' or "system of point charges'') $(\zeta_j)_1^{n}\in\C^n$ picked randomly with respect to the Boltzmann-Gibbs law,
\begin{equation*}
d\mathbf{P}_n(\zeta_1,\cdots,\zeta_n)=\frac{1}{Z_n} e^{-H_n(\zeta_1,\cdots,\zeta_n) }\, dA^{\otimes n}(\zeta_1,\ldots,\zeta_n).
\end{equation*}
Here $Z_n$ is a suitable constant and $H_n$, the energy of the system, is
\begin{equation*}
H_n(\zeta_1,\cdots,\zeta_n)= \sum_{j\ne k}\log \frac{1}{\babs{\zeta_j-\zeta_k}}+n\sum_{j=1}^{n}V_n(\zeta_j).
\end{equation*}

Given a point $\eta\in\C$ and a number $\epsilon>0$, we let $N_\epsilon(\eta)$ denote the number of $\zeta_j$ which fall within distance $\epsilon$ from $\eta$ and define the one-point intensity function (in external potential $V_n$) by
\begin{equation*}
\mathbf{R}_n(\eta) = \lim_{\epsilon\to 0}\frac{\mathbf{E}_n(N_\epsilon(\eta))}{\epsilon^2},
\end{equation*}
where $\mathbf{E}_n$ is expectation with respect to $\mathbf{P}_n$.

Write $d\mu_n=e^{-nV_n}\, dA$ and denote by $\calP_n$ be the subspace of $L^2_a(\mu_n)$ consisting of polynomials of degree at most $n-1$. We will denote by $\bfk_n$ the reproducing kernel for the space $\calP_n$. By a
well-known calculation, given (for particle systems on a line) in \cite[Section IV.7.2]{ST}, we have
$$\bfR_n(\zeta)=\bfk_n(\zeta,\zeta)e^{-nV_n(\zeta)}.$$

Let $\sigma_n$ be the measure $d\sigma_n=\Lap V_n\cdot\chi_S\, dA$.
We define the \textit{microscopic scale}
$r_n$ at $0$ to be the radius so that
$$n\sigma_n(D(0,r_n))=1.$$
(Here and henceforth, $D(p,r)$ is the open disc with center $p$ and radius $r$.)

Since $\Lap Q=(1+O(n^{-1/2k}))\Lap Q_0$ on $D(0,r_n)$, we find that
$$r_n=\tau_0(1+c)^{1/2k}n^{-1/2k}(1+O(n^{-1/2k})),\quad (n\to\infty),$$
where $\tau_0$ is the
"modulus" at the point $0$, i.e., the positive number such that
$$\tau_0^{-2k}=\frac 1 {2\pi k}\int_0^{2\pi}\Lap Q_0(e^{i\theta})\, d\theta.$$
Multiplying $Q$ by a suitable constant, we can assume that
\begin{equation}\label{modo1}nr_n^{2k}=1.\end{equation}
We shall assume \textit{throughout} that \eqref{modo1} holds.

\begin{rem} In the limit as $n\to\infty$, the normalization \eqref{modo1} corresponds to
the condition $\int_{|z|\le 1}\Lap V_0=1$, i.e.,
\begin{equation}\label{modo}
\frac {\Lap^k Q_0(0)}{k[(k-1)!]^2}=
1+c.\end{equation}
\end{rem}

We now rescale about the point $0$ by
\begin{equation}\label{scaling}z=r_n^{-1}\zeta.\end{equation}
Write $R_n(z)$ for the \textit{rescaled one-point function}
$$R_n(z)=r_n^2\bfR_n(\zeta).$$
The family $\{R_n\}$ has a useful compactness property.

\begin{mth}\label{MT1.5} Each subsequence of $\{R_n\}$ has a further subsequence which converges in $L^1_{\operatorname{loc}}(\C)$ and
 locally uniformly on $\C^*$ to a function $R$.
\end{mth}

We will refer to a limit $R=\lim R_{n_k}$ in Theorem \ref{MT1.5} as a \textit{(microscopic) density} (called "limiting one-point function'' in \cite{AKM,AS}).

\begin{mth} \label{MT2} There exists a constant $\con=\con[Q,0]>0$ such that each density $R$ at $0$ satisfies
\begin{equation*}R(z)=\Lap Q_0(z)\cdot (1+O(e^{-\con|z|^{2k}}))\quad \text{as}\quad z\to\infty.\end{equation*}
\end{mth}

\begin{proof}[Proof that Theorem \ref{MT2} implies Theorem \ref{MT1}] We shall show in Section \ref{vogg} that each Bergman function $R_0$ in Theorem \ref{MT1} equals to a density $R$
 (see Theorem \ref{final} below). To prove Theorem \ref{MT1} it thus suffices to appeal to the asymptotics in Theorem \ref{MT2}.
\end{proof}

\begin{rem} In Section \ref{vogg}, it will be seen that $R\le R_0$ and $R_0(z)\sim |z|^{2c}$ as $z\to 0$.
In particular $R(0)=0$ if $c>0$. (See Figure 1.)
\end{rem}

It is sometimes convenient to allow $Q_0$ to be homogeneous of degree $2\lambda$
where $\lambda>0$ is not necessarily an integer. (We also require $Q_0$ to be real-analytic in $\C^*$.)
In this case, a suitable microscopic scale is $r_n=n^{-1/2\lambda}$. Correspondingly, we can consider functions $Q=Q_0+\re H+Q_1$
where $H$ is a polynomial of degree at most $2\lambda$ and where $Q_1$ is real-analytic
in a punctured neighbourhood of the origin and satisfying $Q_1=O(|\zeta|^{2(\lambda+\eps)})$ and $\Lap Q_1=O(|\zeta|^{2(\lambda-1+\eps)})$ as $\zeta\to 0$
where $\eps>0$. Consider an associated potential $V_n$ of the form \eqref{gnoll}.

\begin{mth}\label{extm} Suppose that $\lambda>0$. Then on replacing "$k$" by "$\lambda$", our above results remain in force with respect to the above class of potentials.
In particular we have for each density $R$ that $R(z)=O(|z|^{2c})$ as $z\to0$ and $R(z)=\Lap Q_0(z)\cdot (1+O(e^{-\con|z|^{2\lambda}}))$ as $z\to\infty$.
\end{mth}

{
The proof is by adaptation of the case of integer $\lambda$, see Section \ref{expf}.}

\subsection{Comments} We now explain the roles played by the polynomial $Q_0$ and the parameters $c$, $k$ appearing in
\eqref{vdef}.

Consider the $n$-dependent conformal metric
$$ds_n^{\,2}(\zeta)=e^{-n V_n(\zeta)}\, |d\zeta|^2.$$
Rescaling as in \eqref{scaling},
one obtains the microscopic version $ds^2(z)=e^{- V_0(z)}\, |dz|^2$.
By reference to the latter metric, we say that the origin is a:
\begin{enumerate}[label=(\alph*)]
\item \label{o1} \textit{regular bulk point} if $k=1$ and $c=0$,
\item \label{o2} \textit{bulk singularity caused by vanishing equilibrium density} if $k\ge 2$,
\item \label{o3} \textit{conical singularity} (with total angle $2\pi(1+c)$) if $c\ne 0$.
\end{enumerate}

We can alternatively interpret $R(z)$ as the microscopic one-point intensity in external potential $Q$,
under insertion of
a charge of strength $c$ at the origin.

The figure below shows a few examples of radially symmetric
singularities.

\begin{figure}[ht]\label{fig1}
\begin{center}

\includegraphics[width=.3\textwidth]{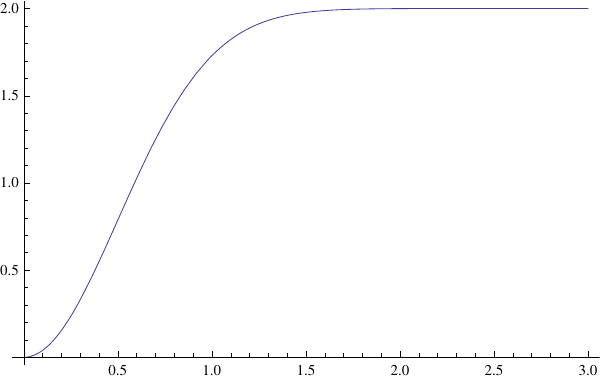}
\hspace{.025\textwidth}
\includegraphics[width=.3\textwidth]{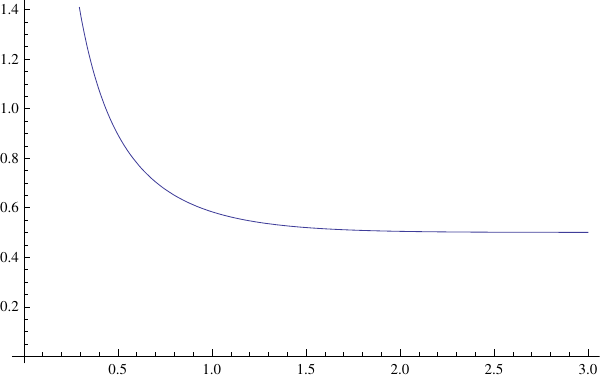}
\hspace{.025\textwidth}
\includegraphics[width=.3\textwidth]{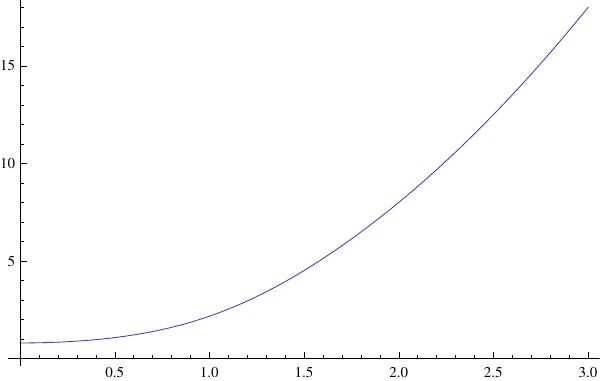}
\end{center}
{
\caption{The Bergman function $R_0$ as a function of positive reals,
for $V_0=2|z|^2-2\log|z|$, $V_0=|z|^2/2+\log|z|$, and $V_0=|z|^4/2$, respectively.}}
\end{figure}

Note that one can superpose the types \ref{o2} and \ref{o3}, thus obtaining "combined'' singularities.

From the applied point of view \cite{AKS} the point of our present estimates for $\bfR_n(\zeta)$ is similar to what is done in the paper \cite{AKM}: we require a rough knowledge of the leading term in $n$, suitable for insertion into the microscopic version of the loop equation, namely the distributional Ward's equation in \cite{AKS}.
Conical singularities have appeared in several other investigations, such as in \cite{HT,KM,LCCW,LY,WW}.

Our methods below
generalize techniques from the papers
\cite{AKM,AS}, which are related to the well-known Tian-Catlin-Zelditch expansion for Bergman kernels.
See \cite{A2} and references for a brief account in the present context.

While we shall not elaborate this point, we note that asymptotic for the Bergman function is related to questions of sampling and interpolation in spaces of entire functions -- $L^2_a(\mu_0)$ in our case. The theory of sampling and interpolation in the classical Fock spaces $F^p_\alpha$ has been well investigated, see \cite{Z}. In \cite[Section 5]{FGHKR}, it is observed that asymptotic for the Bergman function should play a role in the analysis for non-standard weights, such as the ones considered here, cf. also \cite{D,MMOC,MOC}.

\section{A preliminary estimate}  \label{Sec_AP1}
The purpose with this section is to illuminate some of our key constructions in a slightly simplified context, with a "static" weight.
The same method will be elaborated
in the next section for the case of a "varying" ($n$-dependent) weight.
We start by fixing some notation.

A continuous function $h(z,w)$ is called \textit{Hermitian} if $h(z,w)=\overline{h(w,z)}$.
$h$ is called \textit{Hermitian-analytic} (\textit{Hermitian-entire}) if $h$ is Hermitian and analytic (entire) in $z$ and $\bar{w}$.

Given $Q_0$ and $c$ as in \textsection \ref{micropot}, we write $Q_0(z,w)$ for the Hermitian-analytic polynomial such that $Q_0(z,z)=Q_0(z)$, and we define
$$V_0(z,w):=Q_0(z,w)-c\log|z|-c\log|w|.$$

We define an "approximate kernel'' $L_0^\sharp$ by
$$L_0^\sharp(z,w)=
\left[\d_z\dbar_w Q_0\right](z,w)\cdot e^{\,V_0(z,w)}.$$
Write $L_z^\sharp(w)$ for $L_0^\sharp(w,z)$ and, for suitable functions $u$,
$$\pi^\sharp u(z)=\langle u,L_z^\sharp\rangle_{L^2(\mu_0)}=\int_\C u \bar{L}_z^\sharp \, d\mu_0,$$
where $d\mu_0(w)=e^{-V_0(w)}\, dA(w).$

Below we fix a point $z\ne 0$.

\begin{lem} \label{mend} There are positive numbers $m$ and $\num$
such that
$$|1-w/z|<m\quad \Rightarrow\quad 2\re V_0(z,w)\le V_0(z)+V_0(w)-\num|z|^{2k-2}|w-z|^2.$$
\end{lem}

\begin{proof}
Note that $2\re V_0(z,w)-V_0(z)-V_0(w)=2\re Q_0(z,w)-Q_0(z)-Q_0(w)$. The proof now follows by applying Taylor's formula to $Q_0$ in the right hand side.
\end{proof}

Given $z$ with $z\ne 0$ we will consider the annular region consisting of all $w$ such that
$$\tilde{\rho}:=m|z|/2<|w-z|<\rho:=m|z|.$$
Choosing $m$ small enough, we may assume that, throughout this region
$$c_1<|w/z|<c_2$$ for some positive constants $c_1,c_2$.

We now fix a smooth function $\psi_z$ which equals to $1$ on $D(z,\tilde{\rho})$ and to
$0$ outside $D(z,\rho)$, and $\|\dbar\psi_z\|_{L^2}\le \const$ Finally, we define
$$\chi_z(w):=|z/w|^c\psi_z(w).$$

\begin{lem} \label{uv} There is a constant $C$ depending only on $m$ such that for all functions $u\in L^2_a(\mu_0)$ and all $z\ne 0$ we have
$$|\,u(z)-\pi^\sharp[\chi_zu](z)\,|\le C
|z|^{1-k}\|\,u\,\|_{L^2(\mu_0)}\,e^{\,V_0(z)/2}.$$
\end{lem}

\begin{proof} The definition of $L_z^\sharp$ is set up so that
\begin{equation}\label{pa:1}\begin{split}\pi^\sharp[\chi_zu](z)&=\int_\C\chi_z(w)u(w)[\d_z\dbar_wQ_0](z,w)\cdot e^{\,V_0(z,w)-V_0(w)}\, dA(w)\\
&=-\int_\C\frac {u(w)\psi_z(w)F(z,w)}{w-z} \dbar_w[e^{\,Q_0(z,w)-Q_0(w)}]\, dA(w),\\
\end{split}
\end{equation}
where
\begin{equation*}F(z,w)=\frac {(w-z)[\d_z\dbar_w Q_0](z,w)}{\dbar_wQ_0(w,w)-\dbar_wQ_0(z,w)}.\end{equation*}
By Taylor's formula, $F$ obeys
\begin{equation}\label{bestims}F(z,w)=1+O(w-z),\quad \dbar_w F(z,w)=O(w-z),\quad (w\to z).\end{equation}

We now integrate by parts in the identity \eqref{pa:1}, with the result that
$$\pi^\sharp[\chi_z u](z)=u(z)+\epsilon_1+\epsilon_2$$ where
\begin{align*}\epsilon_1&=\int_{|w-z|\ge\tilde{\rho}} \frac {u(w)\cdot\dbar\psi_z(w)\cdot F(z,w)}{w-z}
\frac {e^{V_0(z,w)-V_0(w)}}{|w/z|^c}\, dA(w),\\
\epsilon_2&=\int\frac {u(w)\cdot \chi_z(w)\cdot\dbar_w F(z,w)}{w-z}e^{V_0(z,w)-V_0(w)}\, dA(w),
\end{align*}
where we used that $\dbar\psi_z=0$ on $D(z,\tilde{\rho})$.

By the estimates \eqref{bestims}
we obtain that
\begin{align*}\babs{\epsilon_1}&\le C\tilde{\rho}^{-1}\int_{|w-z|\ge \tilde{\rho}}\babs{u(w)}|\dbar\psi_z(w)|e^{\re V_0(z,w)-V_0(w)}\, dA(w),\\
\babs{\epsilon_2}&\le C\int\chi_z(w)\babs{u(w)}e^{\re V_0(z,w)-V_0(w)}\, dA(w).
\end{align*}
For $w$ in the support of $\chi_z$ we have, by Lemma \ref{mend}, that
\begin{equation}\label{t11}e^{\re V_0(z,w)-V_0(w)/2}\le Ce^{V_0(z)/2-\num|z|^{2k-2}|z-w|^2}.\end{equation}
This gives
\begin{align*}\babs{\epsilon_1}e^{-V_0(z)/2}&\le C\tilde{\rho}^{-1}e^{-\num \babs{z}^{2k-2}\tilde{\rho}(z)^2}
\int\babs{u(w)}|\dbar\psi_z(w)| e^{-V_0(w)/2}
\, dA(w)\\
& \le C'|z|^{-1}e^{-\num (m^2/4)|z|^{2k}}
\|u\|_{L^2(\mu_0)}\|\dbar\psi_z\|_{L^2}.
\end{align*}
By \eqref{t11}, we have that
\begin{align*}|\epsilon_2|e^{-V_0(z)/2}&\le C\|u\|_{L^2(\mu_0)}(\int e^{-2\num|z|^{2k-2}|w-z|^2} dA(w))^{1/2}\\
&\le
C_1|z|^{1-k}\|u\|_{L^2(\mu_0)},\end{align*}
where we used that $\int e^{-t|w|^2}\, dA(w)=t^{-1}$ for $t>0$.
\end{proof}

\begin{rem} Lemma \ref{uv} can be generalized to the case when $k=\lambda$ is an arbitrary positive real number, i.e.,
when $Q_0$ is homogeneous of degree $2\lambda$. In this case we extend $Q_0$ to a Hermitian-analytic function $Q_0(z,w)$ in a
 neighbourhood of a point $z=z_0\ne 0$, so that $Q_0(z,z)=Q_0(z)$ for $z$ near $z_0$. (This can be done by virtue
 of the real analyticity of $Q_0$.) Lemma \ref{uv} generalizes to
this setting, since the proof only involves values of $Q_0(z,w)$ for fixed
$z\ne 0$ and $w$ in a suitable neighbourhood of $z$.
\end{rem}

\section{Asymptotics for microscopic densities} \label{Sec_AP2}

In this section, we prove
Theorem \ref{MT2} modulo an apriori estimate for micro-densities, which we postpone to Theorem \ref{lem:normal}.

Recall first the form of the $n$-dependent potential
\begin{equation}\label{bud}V_n(\zeta)=Q(\zeta)-\frac {2c} n\log|\zeta|-{\frac {1} n h(\zeta)},\end{equation}
where
$Q(0)=h(0)=0$. Write $d\mu_n=e^{-nV_n}\, dA$.

In the following we fix a point $\zeta$ in the annular region
\begin{equation}\label{annular}r_n\le|\zeta|\le r_n\log n.\end{equation}
We assume below that $n$ is large enough so that none of the singularities $a_j\ne 0$ is within distance $2r_n\log n$ from the origin.

\begin{lem} \label{pointl2} Write  $D_\delta=D(\zeta,\delta)$, $\delta=\sigma|\zeta|$, where $\sigma$ is a constant, $0<\sigma<1/2$.
 Let $u$ be a function of the form $u=gP$, where $g$ is holomorphic and $P$ is smooth and nonvanishing in
some neighbourhood of the closure $\overline{D}_{\delta}$.
Then
$$\babs{u(\zeta)}e^{-nV_n(\zeta)/2}\le  C \delta^{-1}e^{\tau n|\zeta|^{2k}} \|u\|_{L^2(\mu_n)},$$
where the number $\tau>0$ can be chosen proportional to $\sigma^2$.
\end{lem}

\begin{proof} Put $$F(\omega)=\babs{u(\omega)}^2e^{-nV_n(\omega)+na|\omega|^2},$$
where $a$ is any number with $a>\sup_{D_\delta}(\Lap Q{-\Lap( h+2\log |P|)/n})$.
 Then $F$ is logarithmically subharmonic in $D_\delta$. Since $\Lap Q\sim \Lap Q_0$ on $D_\delta$, we may choose $a$ proportional to $|\zeta|^{2k-2}$.

The sub-mean property shows that, for some positive constant $c_1=c_1[Q_0]$,
\begin{align*}\babs{u(\zeta)}^2e^{-nV_n(\zeta)}&\le C\frac 1 {\delta^2}\int_{D_\delta}\babs{u(\omega)}^2e^{-nV_n(\omega)}e^{c_1n|\zeta|^{2k-2}(|\omega|^2-|\zeta|^2)}\,
dA(\omega)\\
&\le   C' e^{3c_1\sigma^{2} n|\zeta|^{2k}} \frac 1 {\delta^2}\int_{D_\delta}\babs{u(\omega)}^2e^{-nV_n(\omega)}\, dA(\omega).
\end{align*}
The proof of the lemma is complete.
\end{proof}

For a point $\zeta$ in the annular region \eqref{annular} and a
 small enough positive constant $m$ we define $\tilde{\rho}=m|\zeta|/2$ and $\rho=m|\zeta|$. We take $m$ small enough that
 $c_1<|\omega/\zeta|<c_2$ when $|\omega-\zeta|<\rho$ for some constants
 $c_1,c_2>0$.

Next we fix a smooth function $\psi_\zeta$
which equals to $1$ on $D(\zeta,\tilde{\rho})$ and to $0$ outside
of $D(\zeta,\rho)$, and obeys $\|\dbar\psi_\zeta\|_{L^2}\le C$. Consider the function
$$\chi_\zeta(\omega):=|\zeta/\omega|^ce^{h(\zeta)/2-h(\omega)/2}\psi_\zeta(\omega).$$

Let
$Q(\eta,\omega)$ be a Hermitian-analytic function defined in a neighbourhood of the origin, satisfying $Q(\eta,\eta)=Q(\eta)$. We shall write
$$V_n(\eta,\omega)=Q(\eta,\omega)-\frac {c} n(\log|\eta|+\log|\omega|){-\frac 1 {2n}(h(\eta)+h(\omega)).}$$

The idea is to adapt the construction preceding section, but with $V_0$ replaced by $nV_n$. To this end, we define the "approximate kernel"
$$\bfL_{n}^\sharp(\zeta,\eta)=n\,\d_\zeta\dbar_\eta Q(\zeta,\eta)\cdot e^{\,nV_n(\zeta,\eta)}.$$
Finally we define, for suitable functions $u$, the operator
$$\pi_{n}^\sharp u(\zeta)=\langle u,\bfL_\zeta^\sharp\rangle_{L^2(\mu_n)},\quad d\mu_n(\eta)=e^{-nV_n(\eta)}\, dA(\eta),$$
where, for convenience, we write $\bfL^\sharp_\zeta$ instead of $\bfL^\sharp_{n,\zeta}$.

\begin{lem} \label{mack}
If $u$ is holomorphic in a neighbourhood of $\zeta$, then
$$|\,u(\zeta)-\pi_{n}^\sharp[\chi_\zeta u](\zeta)\,|\le
M_n(\zeta)
\|\,u\,\|_{L^2(\mu_n)}e^{nV_n(\zeta)/2},$$
where, for some constants $C,\para>0$
\begin{equation}\label{mndef}M_n(\zeta)=C(n^{-1/2}|\zeta|^{1-k}+|\zeta|^{-1}e^{-\para n|\zeta|^{2k}}).\end{equation}
\end{lem}

\begin{proof} The proof is accomplished by adapting the method from Lemma \ref{uv}. Recall that $\zeta$ is fixed
in the annular region \eqref{annular} and note that
\begin{align*}\pi_n^\sharp[\chi_\zeta u](\zeta)&=-\int\frac {u(\omega)\psi_\zeta(\omega)F(\zeta,\omega)} {\omega-\zeta}
\dbar_\omega[e^{-n(Q(\omega,\omega)-Q(\zeta,\omega))}]\, dA(\omega),
\end{align*}
where
$$F(\zeta,\omega)=\frac{(\omega-\zeta)\d_\zeta\dbar_\omega Q(\zeta,\omega)}
{\dbar_\omega Q(\omega,\omega)-\dbar_\omega Q(\zeta,\omega)}.$$
In the last expression one can replace $Q$ by $Q_0$ to negligible terms, whence

\begin{equation}\label{viso}F(\zeta,\omega)=1+O(\zeta-\omega),\quad \dbar_\omega  F(\zeta,\omega)=O(\omega-\zeta),\quad (\omega\to\zeta).\end{equation}
We can now write $\pi_n^\sharp u(\zeta)=u(\zeta)+\epsilon_1+\epsilon_2$ where
\begin{align*}\epsilon_1&=\int\frac {u(\omega)\cdot  \dbar\psi_\zeta(\omega)\cdot F(\zeta,\omega)}
{\omega-\zeta}\frac {e^{n\left(V_n(\zeta,\omega)-V_n(\omega)\right)}}
{|\omega/\zeta|^ce^{h(\omega)/2-h(\zeta)/2}}\, dA(\omega),\\
\epsilon_2&=\int\frac {u(\omega)\cdot\chi_\zeta(\omega)\cdot\dbar_\omega F(\zeta,\omega)}
{\omega-\zeta}e^{n(V_n(\zeta,\omega)-V_n(\omega))}\, dA(\omega).
\end{align*}
It follows from Lemma \ref{mend} that, for $\babs{\,\omega-\zeta\,}\le \rho$,
\begin{equation}\label{taly}e^{n(\re V_n(\zeta,\omega)-V_n(\omega)/2)}\le Ce^{\,nV_n(\zeta)/2-\num n|\zeta|^{2k-2}|\zeta-\omega|^2},\end{equation}
where $\num$ is a positive constant.
Inserting the estimates in \eqref{viso} and \eqref{taly}, using also that $\dbar\psi_\zeta=0$ on $D(\zeta,\tilde{\rho})$
we find that (with a suitable $\para>0$)
\begin{align*}\babs{\,\epsilon_1\,}e^{-nV_n(\zeta)/2}&\le C|\zeta|^{-1}e^{-\para n|\zeta|^{2k}}\int\babs{\,u(\omega)\,}|\,\dbar\psi_\zeta(\omega)\,|e^{-nV_n(\omega)/2} \, dA(\omega),\\
\babs{\,\epsilon_2\,} e^{-nV_n(\zeta)/2}&\le C\int\chi_\zeta(\omega)\babs{\,u(\omega)\,}e^{-nV_n(\omega)/2-\num n|\zeta|^{2k-2}|\zeta-\omega|^2}\, dA(\omega).
\end{align*}
Estimating the right hand sides by means of the Cauchy-Schwarz inequality, using that $\int e^{-t|w|^2}\, dA(w)=t^{-1}$, we conclude that
$$(\babs{\,\epsilon_1\,}+\babs{\,\epsilon_2\,})e^{-nV_n(\zeta)/2}\le C(|\zeta|^{-1}e^{-\para n|\zeta|^{2k}}+n^{-1/2}|\zeta|^{1-k})\|\,u\,\|_{L^2(\mu_n)}.$$
The proof is complete.
\end{proof}

We choose $u(\eta)=\bfk_{n}(\eta,\zeta)$ and observe that (see e.g. \cite[p. 30]{AS})
$$|\bfk_n(\zeta,\zeta)-\pi_n[\chi_\zeta \bfL_\zeta^\sharp](\zeta)|=
|\bfk_n(\zeta,\zeta)-\pi_n^\sharp[\chi_\zeta\bfk_{n,\zeta}]|,$$
where
$\pi_n:L^2(\mu_n)\to\calP_n$ is the orthogonal projection, i.e.,
$$\pi_n u(\zeta)=\langle u,\bfk_{n,\zeta}\rangle_{L^2(\mu_n)}.$$
By Lemma \ref{mack},
\begin{equation}\label{hk1:eq}|\,\bfk_n(\zeta,\zeta)-\pi_n[\chi_\zeta \bfL_\zeta^\sharp](\zeta)\,|
\le M_n(\zeta)\sqrt{\bfk_n(\zeta,\zeta)}\cdot e^{\,nV_n(\zeta)/2}.\end{equation}

We next want to estimate the number $\pi_n[\chi_\zeta \bfL_\zeta^\sharp](\zeta)$. For this purpose, we will estimate the norm-minimal solution to a $\dbar$-equation.

Recall that $\zeta$ is fixed in the annular region \eqref{annular}. We will write $\Poly(n)$ for the linear space of analytic polynomials of degree at most $n-1$.

\begin{lem} \label{lem:he}
There exists an element $v_0\in L^2(\mu_n)$ satisfying
$$\begin{cases}&\dbar v_0=\dbar(\chi_\zeta \bfL_\zeta^\sharp), \quad v_0-\chi_\zeta \bfL_\zeta^\sharp\in\Poly(n),\cr

& \cr
 &\|\,v_0\,\|_{L^2(\mu_n)}\le Cn^{-1/2} |\zeta|^{1-k}\|\,\dbar[\chi_\zeta \bfL_\zeta^\sharp]\,\|_{L^2(\mu_n)}.\cr
\end{cases}$$
\end{lem}
\begin{proof}
Let $\check{Q}$ be the obstacle function, i.e., $\check{Q}(\zeta) = -2U^{\sigma}(\zeta)+\gamma$ where $U^{\sigma}$ is the logarithmic potential of the equilibrium measure $\sigma$ and $\gamma$ is the constant which makes $Q = \check{Q}$ on $S$. (See \cite{ST}.) As is well-known
$\check{Q}$ is $C^{1,1}$-smooth in $\C$, harmonic outside $S$, and $\check{Q}\sim 2\log|\zeta|+O(1)$ as $\zeta\to\infty$.

Consider now the modification of $\check{Q}$ defined by $$\phi_n(\zeta)=\check{Q}(\zeta)+\frac{\alpha}{n}\log(1+|\zeta|^2) - \frac{1}{n}(h_0(\zeta)+c\log|\zeta|^2 + \sum_{j=1}^{s}c_j \log|\zeta-a_j|^2 ).$$
In view of the condition \eqref{slow}, we can, by choosing $\alpha$ large enough, make sure that
that (i)
all polynomials of degree $n-1$ are square-integrable with respect to the measure $e^{-n\phi_n} dA$, (ii) $\phi_n$ is strictly subharmonic on the open subset $X=\C\setminus \{0,a_1,\cdots, a_s\}$ of $\C$.

Put $d\mu'_n=e^{-n\phi_n}dA$ and write $\pi'_n$ for the orthogonal projection from $L^2(\mu'_n)$ to the subspace $\calP'_n$ of $L^2(\mu'_n)$ consisting of holomorphic polynomials of degree at most $n-1$.

Now let $v_0 = f - \pi_n'f$ where $f=\chi_\zeta \bfL_\zeta^\sharp$. Note that $\supp f\subset X\cap\Int S$.

Applying the $\dbar$-estimate in \cite[Section 4.2]{H} we have
\begin{equation}\label{eqh1}\|v_0\|^2_{L^2(X,\,\mu_n')}\leq \int_{X} |\dbar f|^2 \,\frac{e^{-n\phi_n}}{n\Lap \phi_n}dA.\end{equation}
Since
$\Lap \phi_n = [1+o(1)]\cdot \Lap Q_0 + O(n^{-1})$ on the support of $\dbar f$, we have by \eqref{eqh1}
$$\|v_0\|_{L^2(\mu_n')}\leq C n^{-1/2}|\zeta|^{1-k}\|\dbar f\|_{L^2(\mu_n)}.$$
Since $n\phi_n \leq nV_n + \const$ we infer that $\|v_0\|_{L^2(\mu_n)}\leq C\|v_0\|_{L^2(\mu_n')}$.
\end{proof}

\begin{rem} If $Q_0$ is homogeneous of degree $2\lambda>0$, then $Q$ is locally of Sobolev class $W^{2,p}$ when $p>1$ is close enough to $1$.
In this case the obstacle function $\check{Q}$ is not $C^{1,1}$-smooth, but merely of local class $W^{2,p}$ near $0$, see \cite[Section 3]{HM}.
However, this is enough to make our above argument work, so the result of Lemma \ref{lem:he} holds in this case as well.
\end{rem}

\begin{lem} \label{bp1} For all $\zeta$ in the annulus \eqref{annular}
we have the estimate
$$|\,\pi_n[\chi_\zeta\bfL_\zeta^\sharp](\zeta)-n\Lap Q(\zeta)\,e^{\,nV_n(\zeta)}\,|\le N_n(\zeta)\, e^{\,nV_n(\zeta)} $$
where, for some constants $C,d>0$,
\begin{equation}\label{nndef}N_n(\zeta)=C \sqrt{n}|\,\zeta\,|^{\,k-2}e^{-{d} n|\zeta|^{2k}} .\end{equation}
\end{lem}

\begin{proof} Let
$u=\chi_\zeta \bfL_\zeta^\sharp-\pi_n[\chi_\zeta \bfL_\zeta^\sharp]$
be the norm-minimal solution in $L^2(\mu_n)$ to the problem $\dbar u=\dbar f$ where $f=\chi_\zeta \bfL_\zeta^\sharp$.
By Lemma \ref{lem:he},
\begin{equation}\label{wilp}\|\,u\,\|_{L^2(\mu_n)}\le Cn^{-1/2}|\zeta|^{1-k}\|\,\dbar[\chi_\zeta \bfL_\zeta^\sharp]\,\|_{L^2(\mu_n)}.\end{equation}
By Lemma \ref{mend} we have
when $\omega\in D(\zeta,\rho)$
$$e^{\re nV_n(\zeta,\omega)-nV_n(\omega)/2}\le Ce^{nV_n(\zeta)/2\,-\num n|\zeta|^{2k-2}\,\babs{\,\omega-\zeta\,}^{\,2}/2}.$$
This gives
\begin{align*}|\,\dbar u(\omega)\,|^{\,2}e^{-nV_n(\omega)}&=|\dbar\chi_\zeta(\omega)|^2 n^2  |\d_\zeta\dbar_\omega V_n(\zeta,\omega)|^2 { e^{n(2\re V_n(\zeta,\omega)-V_n(\omega))}}\\
&\le C(n\Lap Q_0(\zeta))^{\,2}|\,\dbar\chi_\zeta(\omega)\,|^{\,2}e^{nV_n(\zeta)-\num (m/2)^2n|\zeta|^{2k}}.
\end{align*}
By the homogeneity of $\Lap Q_0$ we obtain the estimate
$$\|\,\dbar (\chi_\zeta \bfL_\zeta^\sharp)\,\|_{L^2(\mu_n)}\le Cn\babs{\,\zeta\,}^{\,{2k-2}}e^{-\para n|\zeta|^{2k}}e^{nV_n(\zeta)/2},$$
with a suitable $\para>0$.
Applying \eqref{wilp}, we now get
\begin{equation}\label{boo}\|\,u\,\|_{L^2(\mu_n)}\le C\sqrt{n}|\,\zeta\,|^{\,k-1}e^{-\para n|\zeta|^{2k}}e^{nV_n(\zeta)/2}.\end{equation}

In the notation of Lemma \ref{pointl2} we now put $\delta=\tilde{\rho}=\sigma|\zeta|$. By choosing $\sigma$ small enough, we insure that
the constant $\tau\propto \sigma^{2}$ satisfies $\tau<\nu$.
We also take $g=\bfL_\zeta^\sharp$ and $P=\chi_\zeta$. By Lemma \ref{pointl2} and the estimate in \eqref{boo}, we infer that
$$\babs{u(\zeta)}e^{-nV_n(\zeta)/2}\le
C\sqrt{n}|\zeta|^{k-1}e^{nV_n(\zeta)/2}\cdot |\zeta|^{-1}e^{-dn|\zeta|^{2k}},$$
where $d=\nu-\tau$ is positive.
\end{proof}

We now finish our argument for Theorem \ref{MT2}. Fix $\eps>0$ and take
$\zeta$ in the annular region \eqref{annular}.
By \eqref{hk1:eq} and Lemma \ref{bp1} we have that
\begin{align*}\babs{\,\bfR_n(\zeta)-n\Lap Q_0(\zeta)\,}&\le M_n(\zeta)
\sqrt{\bfR_n(\zeta)}+N_n(\zeta),
\end{align*}
where the functions $M_n$ and $N_n$ are defined in \eqref{mndef} and \eqref{nndef}, respectively.

Multiplying through by $r_n^{\,2}$ and writing $R_n(z)=r_n^{\,2}\,\bfR_n(\zeta)$, $z=r_n^{-1}\zeta$, we get
\begin{equation*}\babs{\,R_n(z)-\Lap Q_0(z)\,}\le r_nM_n(r_nz)\sqrt{R_n(z)}+r_n^2N_n(r_nz).\end{equation*}

A calculation shows that
\begin{equation*}\begin{cases}r_nM_n(r_nz)&=C(r_n^2|z|^{1-k}+|z|^{-1}e^{-\para|z|^{2k}}),\cr
&\cr
r_n^2N_n(r_nz)&=C|z|^{k-2}e^{-d|z|^{2k}}.\cr
\end{cases}\end{equation*}
Letting $R=\lim R_{n_k}$, we get for $|z|\ge 1$
{(assuming that $d\le \para$)
\begin{equation}\label{jimp}|R(z)-\Lap Q_0(z)|\le (C_1|z|^{-1}\sqrt{R(z)}+C_2|z|^{k-2})e^{-d|z|^{2k}}.\end{equation}
Let $M$ be a large constant, and assume that
\begin{equation}\label{bibim}|R(z)-\Lap Q_0(z)|\ge M\Lap Q_0(z)e^{-d|z|^{2k}/2},\quad (|z|\ge C),
\end{equation} where $C$ is large. Then
\eqref{jimp} gives
\begin{equation}\label{con1}R(z)\ge M'e^{d|z|^{2k}}\end{equation}
where $M'=M'(C,M)$ is a new constant. However, by the estimate
of Theorem \ref{lem:normal} below, we have the bound
\begin{equation}\label{con2}R(z)\le B|z|^{\max(4k-2,0)},\qquad (|z|\ge 1).\end{equation}

The estimates \eqref{con1} and \eqref{con2} contradict each other, so the assumption \eqref{bibim}
must be false when $M$ is large enough. We have shown that
$$R(z)=\Lap Q_0(z)(1+O(e^{-\alpha|z|^{2k}})),\quad (|z|\to\infty),$$
where $\alpha=d/2$.} Hence Theorem \ref{MT2} follows once we have proved the estimate \eqref{con2}; the proof of this is carried out in the succeeding section. q.e.d.

\section{Construction of microscopic densities}\label{vogg}

In this section, we prove Theorem \ref{MT1.5} on the existence of micro-densities. In the process, we will supply the
apriori bounds on the 1-point function which are needed to complete our proof of Theorem \ref{MT2}.

Towards this end, let {$Q=Q_0+\re H+Q_{1}$} be the canonical decomposition.
 We consider, as before, the potential $$ V_n=Q-(2c/n)\ell-(1/n) h,\quad (\ell(\zeta)=\log\babs{\zeta}).$$
We assume, as always, that $Q(0)=h(0)=H(0)=0$ and that $nr_n^{2k}=1$.
Recall that $d\mu_n=e^{-nV_n}\, dA$.

\subsection{Limiting kernels}

We introduce the "rescaled potential" $\tilde{V}_n$ by
$$\tilde{V}_n(z)=nQ(r_n z){-h(r_nz)}-2c\log |z|,$$
and accordingly we define the measure $d\tilde{\mu}_n=e^{-\tilde{V}_n}dA$.

Observe that the reproducing kernel $k_n$ for the subspace $\tilde{\calP}_n=\Poly(n)\subset L^2_a(\tilde{\mu}_n)$ is given by
\begin{equation}\label{tilp}k_n(z,w)=r_n^{2+2c}\bfk_n(\zeta,\eta),\quad  \zeta=r_nz,\, \eta=r_nw,\end{equation}
where $\bfk_n$ is the reproducing kernel for $\calP_n$ (see Section \ref{mde}).

We are heading for a normal families argument, based on estimates for the function
$k_n(z,z)$. For this purpose, we now prove two lemmas.

\begin{lem}\label{lem:est} Assume that $-1< c \leq 0$ and $k>0$ (not necessarily an integer).
Then for each $T\ge 1$ there is a constant $C=C(T)$ such that for all $u\in L^2_a(\tilde{\mu}_n)$ we have
 \begin{equation}\label{fe1}
\babs{u(z)} e^{-nQ(r_n z)/2}\leq C \norm{u}_{L^2(\tilde{\mu}_n)},\quad |z|\le T.
\end{equation}
Moreover, there is a constant $B$ independent of $T$ such that
\begin{equation}\label{fe2}C(T)\le BT^{\max(2k-1,0)-c}.\end{equation}
\end{lem}

\begin{proof} Write $D_r=\{|z|\le r\}$ and $D_r^*=D_r\setminus\{0\}$ and fix a number $\delta$ with $0<\delta<1/2$.
We also fix a constant $\alpha$ such that
\begin{equation}\label{ale}\alpha > \max_{|z|=1}\frac {\Lap Q_0(z)}{k^2}\end{equation}

Now consider the function
$$F_n(z)=\babs{u(z)}^2 e^{-\tilde{V}_n(z)+\alpha\babs{z}^{2k}}.$$
For $z\in D^*_{T+\delta}$ we have
$$\Lap\tilde{V}_n(z) = \Lap Q_0(z) + O(r_n),\quad (n\to\infty)$$
so for $n$ large enough we will have
$$\alpha k^2|z|^{2k-2}-\Lap\tilde{V}_n(z)=k^2|z|^{2k-2}\left[\alpha -\frac {\Lap Q_0(z/|z|)}{k^2}\right]+O(r_n)$$
which is strictly positive by \eqref{ale}.
It follows that $F_n$ is logarithmically subharmonic on $D^*_{T+\delta}$. We will use this
to estimate the left hand side of \eqref{fe1} for a fixed $z\in D_T$.

First we fix a real number $k$ with $k>1/2$ and assume that $\delta\le |z|\le T$. Since $F_n$ is subharmonic,
\begin{align*}
\babs{u(z)}^2 e^{-\tilde{V}_n(z)+\alpha\babs{z}^{2k}} \leq 4\, \delta^{-2}\int_{D(z,
\delta/2)} \babs{u(w)}^2 e^{-\tilde{V}_n(w)+\alpha\babs{w}^{2k}}dA(w).
\end{align*}
By Taylor's theorem, we have when $|w-z|\le \delta /2$
$$|w|^{2k}-|z|^{2k}\le 2k|z|^{2k-1}\frac \delta 2+k(2k-1)x_*^{2k-2}\left(\frac\delta 2\right)^2$$
for some real number $x_*\in (|z|,|z|+\delta/2)$. If $1/2<k<1$ then the error term has a bound
$$k(2k-1)x_*^{2k-2}(\delta/2)^2\le C\delta^{2k}.$$
If $k\ge 1$ then the error term is bounded by
$$k(2k-1)x_*^{2k-2}(\delta/2)^2\le C\delta^2T^{2k-2}.$$

Now choose $\delta=T^{1-2k}$. Then for $|w-z|\le\delta/2$,
$$e^{\alpha(|w|^{2k}-|z|^{2k})}\le e^{\alpha kT^{2k-1}\delta+C\max\{\delta^{2k},\delta^2T^{2k-2}\}}\le e^{\alpha k+C\max\{T^{2k(1-2k)},T^{-2k}\}}.$$
Hence
$$|u(z)|^2e^{-nQ(r_n z)}\le C_1\norm{u}^2_{L^2(\tilde{\mu}_n)}$$
where $C_1$ is at most
$$C_1\le 4T^{4k-2}e^{\alpha k+C\max\{T^{2k(1-2k)},T^{-2k}\}-2c\log T}:=B^2T^{4k-2c-2}.$$

Now suppose $|z|<\delta$ where $\delta=T^{1-2k}$. In this case we have for $|w-z|\le \delta$
$$|w|^{2k}-|z|^{2k}\le (2\delta)^{2k}$$
so for $|z|< \delta$
\begin{align*}
|u(z)|^2e^{-nQ(r_nz)} &\le \delta^{-2} e^{-\alpha|z|^{2k}}\int_{D(z,\delta)}\babs{{u}(w)}^2 e^{-nQ(r_n w)+\alpha\babs{w}^{2k}}\,dA(w) \\
&\le C_2\|u\|_{L^2(\tilde{\mu}_n)}^2,
\end{align*}
where
$$C_2 \leq C\delta^{-2}e^{\alpha(2\delta)^{2k}}\le C T^{4k-2}e^{\alpha 4^kT^{2k(1-2k)}}\le B^2T^{4k-2c-2}.$$

On the other hand, the above argument can be adapted for the case of $0<k\leq 1/2$. The settings are the same, but we take $\delta$ to be independent of $T$. More precisely, for $z,w$ with $\delta \leq |z| \leq T$ and $|w-z|\leq \delta/2$, we have
$$|w|^{2k} - |z|^{2k} \leq k\,\delta^{2k}$$
by Taylor's theorem. This implies that
$$|u(z)|^2 e^{-nQ(r_n z)} \leq C_1 \|u\|^{2}_{L^2(\tilde\mu_n)}, \quad \delta \leq |z|\leq T$$
where $C_1 \leq 4\delta^{-2} e^{\alpha k \delta^{2k}-2c\log T} \leq B^2 T^{-2c}$ for some constant $B$.
For $z,w$ with $|z|<\delta$ and $|w-z|\leq \delta$, we have
$$|w|^{2k} - |z|^{2k} \leq (2\delta)^{2k},$$ which gives
$$|u(z)|^2 e^{-nQ(r_n z)} \leq C_2 \|u\|^{2}_{L^2(\tilde\mu_n)}, \quad |z|< \delta$$ where $C_2 \leq C \delta^{-2} e^{\alpha (2\delta)^{2k}} \leq B^{2} T^{-2c}$.
\end{proof}

We now consider an arbitrary $c>-1$ and write
\begin{equation}\label{cqdec}c=q-c'\end{equation} where $q$ is a non-negative integer and $0\le c' <1$.

\begin{lem} \label{lem:q} For all $u\in L^2_a(\tilde{\mu}_n)$ and all $T\ge 1$ we have the estimate
\begin{equation}\babs{u(z)}e^{-nQ(r_nz)/2}\le BT^{\max(2k-1,0)+c'}|z|^{-q}\|u\|_{L^2(\tilde{\mu}_n)},\quad (|z|\le T).
\end{equation}
\end{lem}

\begin{proof} Replace "$u(z)$'' by "$z^qu(z)$'' in Lemma \ref{lem:est}.
\end{proof}

Now recall the canonical decomposition $Q=Q_0+\re H+Q_{1}$ and
consider the Hermitian-entire function
\begin{equation}\label{hef}L_n(z,w)=k_n(z,w)e^{-\tilde{H}(z)/2-\overline{\tilde{H}(w)}/2},\quad (\tilde{H}(z)=nH(r_nz)).\end{equation}
Here $k_n$ is the reproducing kernel for the space $\tilde{\calP}_n$, cf. \eqref{tilp}.

We claim that $L_n$ is the reproducing kernel for the Hilbert space
$$\calH_n= \{ f;\, f=p\cdot e^{-\tilde{H}_n/2} \,,\, p\in\Poly(n)\},$$ with the norm of $L^2(\mu_{0,n})$, where
$$d\mu_{0,n}(z)=e^{-Q_0(z)-nQ_{1}(r_n z)+h(r_nz)+2c\log\babs{z}}\,dA(z).$$ Indeed,
for an element $f=p\cdot e^{-\tilde{H}_n/2}\in \calH_n$, we have
\begin{align*}\int f(z) \bar{L}_n(z,w) d\mu_{0,n}(z)&= e^{-\tilde{H}_n(w)/2} \int p(z)\bar{k}_n(z,w) e^{-\tilde{V}_n(z)} \,dA(z)\\
&=f(w).
\end{align*}
We also notice that $\mu_{0,n} \to \mu_0$
in the vague sense of measures, as $n\to\infty$.

In the following, we write
\begin{align}\label{bob}K_n&(z,w)=k_n(z,w)e^{-\tilde{V}_n(z)/2-\tilde{V}_n(w)/2}=L_n(z,w)\\
& \cdot e^{-n(Q(r_nz)+Q(r_nw)-H(r_nz)-\bar{H}(r_nw))/2{+(h(r_nz)+h(r_nw))/2}}\nonumber
|zw|^{c}.\end{align}
Note that the rescaled one-point function $R_n$ is just $R_n(z)=K_n(z,z)$.

\begin{lem}\label{lem:rn} The family $\{R_n\}$ satisfies the following bound, for all $T\ge 1$,
$$R_n(z)\le BT^{\max(4k-2,0)+2c'}|z|^{-2c'},\qquad (|z|\le T),$$
where $B$ is an absolute constant.
\end{lem}

\begin{proof} Since
\begin{equation*}k_n(z,z)=\sup \{\babs{p(z)}^2\, ;\, p\in \tilde{\calP}_n,\,\norm{p}_{L^2(\tilde{\mu}_n)}\leq 1 \},\end{equation*} it follows from Lemma \ref{lem:q} that for all $n$ and all $|z|\le T$,
\begin{equation*}k_n(z,z) \leq BT^{\max(4k-2,0)+2c'}|z|^{-2q}
e^{nQ(r_n z)}.\end{equation*}
Hence
\begin{equation}\label{jod}\begin{split}L_n(z,z)&=k_n(z,z)e^{-\re \tilde{H}(z)}\\
&\le BT^{\max(4k-2,0)+2c'}|z|^{-2q}e^{Q_0(z)+nQ_{1}(r_n z)},\quad (|z|\le T),\\
\end{split}\end{equation}
The lemma follows by applying this estimate for $L_n(z,z)$ in the right hand side of the expression \eqref{bob}.
\end{proof}

 To interpret this result, we introduce a few basic isomorphisms.

Consider $c$ and $c'$ related by \eqref{cqdec} and write
$$V_0(z)=Q_0(z)-2c\log |z|,\quad V_0'(z)=Q_0(z)+2c'\log|z|.$$
Denote $d\mu_0=e^{-V_0}\, dA$, $d\mu_0'=e^{-V_0'}\, dA$. Let us also put $\calH_0=L^2_a(\mu_0)$ and
$\calH_0'=\{z^qu(z);\, u\in L^2_a(\mu_0)\}\subset L^2_a(\mu_0')$. Then the map
$$\calH_0\to\calH_0'\quad ,\quad u(z)\mapsto z^qu(z)$$
is an isometric isomorphism. Hence if we denote $L_0$ and $L_0'$ the Bergman kernels of $\calH_0$ and $\calH_0'$ respectively,
then
$$L_0'(z,w)=(z\bar{w})^q L_0(z,w).$$

Using the canonical decomposition $Q=Q_0+\re H+Q_1$, we similarly define the measures $\mu_{0,n}$ and $\mu_{0,n}'$ by
\begin{equation*}\begin{cases}
d\mu_{0,n}(z)&=e^{-V_0(z)-nQ_1(r_nz){+h(r_nz)}}\, dA(z),\cr
&\cr
d\mu_{0,n}'(z)&=e^{-V_0'(z)-nQ_1(r_nz)
{+h(r_nz)}}\, dA(z).\cr
\end{cases}
\end{equation*}
We write
$$\calH_n=\{f;\, f(z)=p(z)\cdot e^{-nH(r_nz)/2};\, p\in\Poly(n)\}$$
with the norm of $L^2(\mu_{0,n})$ and
$$\calH_n'=\{f;\, f(z)=z^q\cdot p(z)\cdot e^{-nH(r_nz)/2};\, p\in\Poly(n)\}$$
with the norm of $L^2(\mu_{0,n}')$. Obviously, the map
$\calH_n\to\calH_n'$, $u(z)\mapsto z^q u(z)$ is an isometric isomorphism. Hence if we denote by
$L_n$ and $L_n'$ the respective reproducing kernels, we have
$$L_n'(z,w)=(z\bar{w})^qL_n(z,w).$$

\begin{lem}\label{lem:no1} The Hermitian-entire functions $\{L_n'\}$ form a normal family.
\end{lem}

\begin{proof} The estimate \eqref{jod} shows that the family $L_n'(z,z)$ is locally uniformly bounded on $\C$. The lemma follows, since $|L_n'(z,w)|^2 \le L_n'(z,z) L_n'(w,w)$.
\end{proof}

We now come to our main result in this section, which, in particular, implies Theorem \ref{MT1.5}. Before stating the theorem, we require a definition. A Hermitian function $c(z,w)$ is called a \textit{cocycle} if there is a unimodular function $g$ such that $c(z,w)=g(z)\overline{g(w)}$.

\begin{thm}\label{lem:normal}
There exists a sequence of cocycles $c_n$ such that each subsequence of $\{|z\bar{w}|^{c'}(c_n K_n)(z,w)\}_n$ has a further subsequence converging uniformly on compact subsets of $\C^2$ as $n \to\infty$. Moreover, each subsequential limit $K=\lim (c_{n_l} K_{n_l})$ is of the form $K(z,w)=L(z,w)e^{-V_0(z)/2-V_0(w)/2}$ where $L'(z,w)=(z\bar{w})^qL(z,w)$ is Hermitian-entire. If $R(z)=K(z,z)$ then the convergence $R_{n_l}\to R$ holds in $L^1_{\loc}(\C)$,
and for $T\ge 1$,
\begin{equation}\label{mupp}R(z)\le BT^{\max(4k-2,0)+2c'}|z|^{-2c'},\qquad (|z|\le T).\end{equation}
\end{thm}

\begin{proof}
Define a function $E_n(z,w)$ by
$$E_n(z,w)=e^{n(H(\zeta)+\bar{H}(\eta)-Q(\zeta)-Q(\eta))/2{+(h(\zeta)+h(\eta))/2}},\quad (\zeta=r_nz,\quad \eta=r_nw).$$
Then $K_n(z,w)=(L_n E_n)(z,w) |zw|^{c}$, see \eqref{bob}.
Denoting by $c_n$ the cocycle  $$c_n(z,w)=e^{i\im(\tilde{H}(z)-\tilde{H}(w))/2},$$
we have the locally uniform asymptotic relation
$$E_n(z,w)=c_n(z,w)\,e^{-Q_0(z)/2-Q_0(w)/2}(1+o(1)),\quad n\to\infty.$$

By Lemma \ref{lem:no1}, the functions $L_n'(z,w)=(z\bar{w})^qL_n(z,w)$ form a normal
family. Hence each subsequence has a further subsequence (renamed  $L_n'$) which
converges locally uniformly to a Hermitian-entire function $L'$. We put $L(z,w)=(z\bar{w})^{-q}L'(z,w)$.

We have shown that there are cocycles $c_n$ such that
\begin{align*}|z\bar{w}|^{c'}(c_{n}K_{n})(z,w)&=|z\bar{w}|^q (c_nL_{n}E_{n})(z,w)\\
&\to |z\bar{w}|^q L(z,w)\,e^{-Q_0(z)/2-Q_0(w)/2}
\end{align*}
locally uniformly as $n\to\infty$. Write $$K(z,w)=L(z,w) \, e^{-V_0(z)/2-V_0(w)/2}.$$

Let $R_n(z)=K_n(z,z)$ and $R(z)=K(z,z)$, and note that by Lemma \ref{lem:rn}, the family $\{R_n\}$ is locally uniformly integrable. This shows that the convergence $R_n\to R$ holds in $L^1_{\loc}(\C)$.
Finally, the estimate \eqref{mupp} is clear from Lemma \ref{lem:rn}.
\end{proof}

\subsection{Positivity} Theorem \ref{lem:normal} has the following consequence.

\begin{thm} \label{post} Let $L_0'$ be the Bergman kernel of $L^2_a(\mu_0')$ (corresponding to parameter value $c=c'$)
and let $L'=(z\bar{w})^qL$ be a limiting kernel in Theorem \ref{lem:normal}. Then $L_0'-L'$ is a positive matrix:
$\sum_{j,k=1}^N\alpha_j\bar{\alpha}_k(L'_0-L')(z_j,z_k)\ge 0$
for all choices of scalars $\alpha_j\in \C$ and all points $z_j\in \C$.
\end{thm}

\begin{proof} It follows from Theorem \ref{lem:normal} that $L'=\lim L_n'$ where $L_n'$ is the Bergman kernel of the space
$L^2_a(\mu_{0,n}')$. Since $\mu_{0,n}'\to\mu_0'$ vaguely, and since $L_0'$ is the Bergman kernel of the space $L^2_a(\mu_0')$, we can use Fatou's lemma as in \cite{AKM} or \cite{AS} to conclude that $L'$ is the Bergman kernel of
some semi-normed Hilbert space $\calH_*'$ of entire functions, which is contractively embedded in $L^2_a(\mu_0')$.
It hence follows as in \cite{AKM,AS}, using Aronszajn's theorem on differences of reproducing kernels in \cite{Ar}, that $L_0'-L'$ is a positive matrix.
\end{proof}

\begin{cor} \label{bodo} Each density $R$ satisfies $R\le R_0$. Moreover there are constants $C_1,C_2>0$ such that $C_1|z|^{2c}\le R_0(z)\le C_2|z|^{2c}$ for all $z$
in a punctured neighbourhood of the origin.
\end{cor}

\begin{proof} That $R\le R_0$ on $\C^*$ follows from Theorem \ref{post}, since $R(z)=L'(z,z)e^{-V_0'(z)}$ and
$R_0(z)=L_0'(z,z)e^{-V_0'(z)}$. It now suffices to note that $R_0(z)=|z|^{2c}M_0(z)$ where the function $M_0(z)=L_0(z,z)e^{-Q_0(z)}$ is continuous.
$M_0$ is also bounded below in a neighbourhood of the origin; indeed $L_0(z,w)=\sum_0^\infty e_j(z)\bar{e}_j(w)$ where $e_j$ are the orthonormal polynomials
with respect to $\mu_0$, so
$$M_0(0)  \geq  |e_0(0)|^2=(\int_\C e^{-V_0}\, dA)^{-1}>0.$$
\end{proof}

\subsection{The Bergman function as a microscopic density} \label{BFM}
Fix a positive definite homogeneous polynomial $Q_0$ and consider the $n$-dependent potential $$V_n(\zeta)=Q_0(\zeta)-\frac {2c}n\log|\zeta|.$$
We note that the origin is an interior point of the droplet corresponding to the potential $Q_0$.
(This follows by using that $Q_0$ is strictly convex and attains its minimum at $0$, see \cite{HM}.)

Consider now the holomorphic kernels $L_n'=(z\bar{w})^qL_n$, cf. \eqref{hef}.
Recall that this kernel reproduces for the space $\calH_n'$ of polynomials $z^q p(z)$ where $p\in\Poly(n)$ equipped with the norm of $L^2(\mu_0')$. Here $$d\mu_0'(w)=e^{-V_0'(w)}\, dA(w)= e^{-Q_0(w)-2c'\log\babs{w}}\,dA(w).$$
By Theorem \ref{lem:normal}, every subsequence of $\{L_n'\}$ has a subsequence which converges locally uniformly to a Hermitian-entire limit $L'$. As the inclusions $\calH_n'\subset\calH_{n+1}'$ are isometric, the limit must be the same for all subsequences.

Now note that for all polynomials $p$
\begin{align*}
z^qp(z)= \lim_{n\to\infty}\int w^qp(w) L_n'(z,w)\, d\mu_0'(w)= \int w^qp(w) L'(z,w)\, d\mu_0'(w).
\end{align*}
Since polynomials are dense in $L_a^2(\mu_0')$, we have for all $f\in L^2_a(\mu_0')$
$$z^qf(z)= \int w^qf(w) L'(z,w)\, d\mu_0'(w).$$
It follows that $L'=L_0'$ on $\C^2$ where $L_0'$ is the Bergman kernel of $\calH_0'.$

We have shown the following result.

\begin{thm} \label{final} In the situation above, the Bergman function $R_0$ of the space $L^2_a(\mu_0)$ equals to the density $R(z)
=L(z,z)e^{-V_0(z)}$.
\end{thm}

\subsection{Proof of Theorem \ref{extm}} \label{expf} Now consider a potential of the special form $Q=Q_0(\zeta)+\re H(\zeta)+\cdots$ where $Q_0$ is homogeneous of degree $2\lambda$ and the dots represent negligible
terms.
Using the remark at the end of Section \ref{Sec_AP1} and the remark after Lemma \ref{lem:he},
we may repeat our arguments with "$k$" replaced by "$\lambda$".
As a result, our theorems must be true also for generalized potentials of this kind. $\qed$

\end{document}